\documentclass[12pt]{article}
\usepackage{geometry}

\usepackage[pdftex]{graphicx}
\usepackage{amssymb}
\usepackage{epstopdf}
\usepackage{amsthm}
\usepackage{setspace}
\usepackage{caption}
\usepackage{subcaption}
\usepackage[nottoc, numbib]{tocbibind}
\usepackage{tikz}
\usetikzlibrary{positioning}
\theoremstyle{definition}

\newtheorem{thm}{Theorem}                         
\newtheorem*{thm*}{Theorem}  
\newtheorem{fact}{Fact} 
\newtheorem{cor}{Corollary}
\newtheorem*{cor*}{Corollary}
\newtheorem{lem}{Lemma}
\newtheorem*{lem*}{Lemma}
\newtheorem*{conj}{$\Delta^2$ Conjecture}
\newtheorem*{con}{Conjecture}
\newtheorem{defin}{Definition}

\newcommand{\HRule}{\rule{\linewidth}{0.5mm}}

\title{Brief Article}
\begin{document}
%
%
%
%
%
%


\begin{titlepage}

\center 
 

\textsc{\LARGE University of South Carolina}\\[1.5cm] 
\textsc{\Large Senior Thesis}\\[0.5cm] 
\textsc{\large Submitted in partial fulfillment of the requirements for the Bachelor of Science with Distinction in Mathematics}\\[0.5cm] 


\HRule \\[0.4cm]
{ \LARGE \bfseries Graph Labeling with Distance Conditions and the Delta Squared Conjecture}\\[0.4cm] 
\HRule \\[1.5cm]
 

\begin{minipage}{0.4\textwidth}
\begin{flushleft} \large
\emph{Author:}\\
Cole \textsc{Franks} 
\end{flushleft}
\end{minipage}
~
\begin{minipage}{0.4\textwidth}
\begin{flushright} \large
\emph{Advisor:} \\
Dr. Jerry \textsc{Griggs}\\
\emph{Second Reader:} \\
Dr.  Linyuan  \textsc{Lu}\\ 
\emph{Committee Member:} \\
Dr.  L\'{a}zl\'{o} \textsc{Sz\'{e}kely}
\end{flushright}
\end{minipage}\\[3cm]


{\large May, 2013}\\[3cm] 


 

\vfill 

\end{titlepage}

\begin{abstract}
An $L({2,1})$-labeling of a simple graph $G$ is a function $f:V(G) \rightarrow \mathbb{Z}$ such that if $xy \in E(G)$, then $|f(x) - f(y)| \geq 2$ and if the shortest path in $G$ connecting $x$ and $y$ has two edges then $|f(x) - f(y)| \geq 1$. The $L(2,1)$-labeling problem is an example of a Channel Assignment Problem, which aims to model the frequency assignment of transmitters which could interfere when near. The objective is generally to minimize the difference between the highest and lowest label used, called the span. The $L(2,1)$-labeling problem was posed by Jerrold Griggs and Roger Yeh in 1992, and they conjectured that any graph has a labelling with span no greater than the square of $\Delta(G)$, the maximum degree, if $\Delta(G)\geq 2$. We first review some previous results about $L(2,1)$-labelings, and then we cover some methods in hamilton paths and coloring that will be used throughout this thesis. We then prove that  if the order of $G$ is at most $(\lfloor \Delta/2  \rfloor + 1 )(\Delta^2 - \Delta+ 1) - 1$, then $G$ has an $L(2,1)$- labeling with span no greater than $\Delta^2$. This shows that for graphs no larger than the given order, the 1992 ``$\Delta^2$ Conjecture" of Griggs and Yeh holds. In fact, we prove more generally that if $L \geq \Delta^2 +1$, $\Delta\geq 1$ and $$|V(G)| \leq (L - \Delta)\left(\left\lfloor\frac{L-1}{2\Delta} \right\rfloor + 1\right) - 1,$$ then $G$ has an $L(2,1)$-labeling with span at most  $L - 1$. In addition, we show that there is a polynomial time algorithm to find $L(2,1)$-labelings with span and order satisfying the conditions above. We also exhibit an infinite family of graphs with minimum span $\Delta^2 - \Delta + 1$, which is the largest of any known infinite family.																														
\end{abstract}
\newpage
\tableofcontents

\section{Summary}\label{sum}

The \emph{channel assignment problem} is the assignment of channels to stations such that those stations close enough to interfere receive distant enough frequencies. Interference between stations occurs when frequencies are too similar.  The $L(2,1)$-labeling problem is an instance of the channel assignment problem in which closer transmitters would be required to have channels that differed by more than those of the slightly more distant transmitters. In this problem, graphs are used to model networks, and the physical distance between two nodes is modeled by the minimum number of hops between them in the network. A graph is a set of vertices, or points, and edges between these points. In this model, the vertices represent the nodes in the network and there are edges between nodes that are close enough to interfere. An $L(2,1)$-labeling of a graph $G$ is an integer labeling of $G$ in which two vertices with an edge between them must have labels differing by at least $2$, and with a path of length two between them must differ by at least $1$. The goal of the problem is to minimize the span, or the difference between the lowest and highest labels used.\\
In Section \ref{back}, we discuss some other applications of and work on the channel assignment problem, such as Hale's $T$ colorings and scheduling problems. We also discuss generalizations of the $L(2,1)$-labeling problem, such as those in which alternate distances are required to label the transmitters, or those in which specific distances are assigned to edges. In addition, we discuss a conjecture on the relationship between the span of labelings and the maximum degree, or the largest number of edges on any vertex. Many bounds on the span have been proven, but none have reached the $\Delta^2$ conjecture, which claims that any graph $G$ has an $L(2,1)$-labeling with span at most the square of $G$'s maximum degree, as long as the maximum degree.\\
In Section \ref{term}, we review basic graph theory terminology. Among the terms reviewed are concepts such as simple graphs, vertex degree, colorings, paths, hamiltonicity, graph distance, and the square of a graph.\\
In Section \ref{bounds} we visit some previous work on the $L(2,1)$ labeling problem. Previous authors have bounded the span of certain special classes of graphs, such as cycles, paths, and trees. In many cases these bounds are exact or very close. In addition, we discuss the progress of algorithms to find labelings with small span. Several authors have made bounds close to that in the $\Delta^2$ Conjecture, and often the algorithms in these proofs are notable. We will mention some hardness results in this section as well;  the hardness of testing for minimum span is known. 
Section \ref{ham} deals with some results on hamilton paths, or paths through a graph that visit each vertex exactly once, which will be needed in the proofs of our main results. We will discuss some necessary conditions for hamilton paths that depend on the minimum degree of a graph, as well as some that depend on the specific number of edges on each vertex. Here we will also cover how hamilton paths relate to the channel assignment problem. 
Section \ref{main} is concerned with the main results in this paper. In particular, we look at a special case of the weighted graph model which subsumes the $L(2,1)$-labeling problem. We use this to prove that the $\Delta^2$ Conjecture holds if the number of vertices in $G$ is below a certain bound which is cubic in the maximum degree. This technique also improves the current best bound for graphs with slightly larger numbers of vertices. The proof also results in an algorithm  to find labelings with a given size-dependent span. 
In Section \ref{tightnesscomments} we will present some graphs for which the $\Delta^2$ Conjecture is very close, and we find a family of graphs with the highest known minimum span. \\
Finally, Section \ref{future} covers some other questions which have arisen during the course of this research, and some directions the techniques used in this paper could extend. In particular, we examine a version of the $\Delta^2$ conjecture for the weighted graph model, and we ask if similar hamilton cycle techniques apply for a cyclic version of the labeling problem.

\section{Background}\label{back}

The \emph{channel assignment problem} is the determination of assignments of channels (integers) to stations such that those stations close enough to interfere receive distant enough channels. Interference between stations occurs when frequencies are too similar. This is one motivation of classic vertex colorings of graphs; one can form a graph in which vertices are stations and edges are added between stations close enough to interfere. If there is a frequency condition that close stations must obey, then different colors are just integral multiples of this frequency, and a proper coloring corresponds to a valid frequency assignment. Hale \cite{hale} formulated the problem in terms of $T$-colorings, which are integer colorings in which adjacent vertices' colors cannot differ by a member of a set of integers $T$ with $\{0\}\subset T$. This allows the models to take into account certain frequency differences that may be ``bad" in some sense. This problem also applies to scheduling problems; for example, one might want to schedule some set of events $X$ that each consist of two separate one-hour events $x_1$ and $x_2$ that are two hours apart from beginning to beginning. In this case one can form a graph where the vertices are events, and two events that are not allowed to overlap would have an edge between them. In that case, it is enough to assign a certain hour to the time $t_x$ directly between the events $x_1$ and $x_2$. Integer labels are a choice for the hour $t_x$. Then adjacent vertices $x$ and $y$ are labeled correctly if and only if $|t_x - t_y| \notin \{0, 2\}$. Roberts \cite{roberts} proposed a generalization of the channel assignment problem in which closer transmitters would be required to have channels that differed by more than those of the slightly more distant transmitters, adding a condition for non-adjacent vertices as well. The $L(2,1)$-labeling problem was first studied by Griggs and Yeh in 1992 in response to Roberts' proposal. An $L(2,1)$-labeling of a graph $G$ is an integer labeling of $G$ in which two vertices at distance one from each other must have labels differing by at least $2$, and those at distance two must differ by at least $1$. Formally:
\begin{defin} An $L({2,1})$-labeling of a simple graph $G$ is a function $f:V(G) \rightarrow \mathbb{Z}$ such that if $xy \in E(G)$, then $|f(x) - f(y)| \geq 2$ and if the distance between $x$ and $y$ is two then $|f(x) - f(y)| \geq 1$. 
\end{defin}

Note that one may always assume that the labeling starts at zero, because one can subtract a constant from all the labels while preserving the distance conditions.\\

There are natural generalizations of this problem, such as the $L(d_1, d_2)$-labeling problem, in which vertices must be labeled by real numbers so that vertices at distance $1$ differ by at least $d_1$ and those at distance two must differ by $d_2$. For example, the $L(p,1)$-labeling and $L(h, k)$-labeling problems, where $h,k,p \in \mathbb{Z}$, have received particular attention.  Griggs and Yeh actually showed that it suffices to consider labels that are integral multiples of $d$ in the $L(2d, d)$-labeling problem, meaning it is enough to consider $L(2,1)$-labelings \cite{griggs92}.

The goal in each of these problems is to find a labeling with minimum span, or distance between the highest and lowest labels used. This model uses graph distance as a discrete analog of actual distance, but the model does not work perfectly. Say two transmitters are \emph{very close} if they are close enough to require frequencies differing by $2$, and \emph{close} if their frequencies only need to differ by $1$. There could be a scenario, as shown in Figure \ref{transmitter}, in which two transmitters were close but they had no common very close neighbor. Those two vertices would not be at distance two in the corresponding graph, so an $L(2,1)$-labeling might assign them the same label. However, the graph version of the problem can still yield good bounds and heuristics for the realistic labeling problem. In addition, there are some situations in which they are equivalent. If stations are placed at the centers of hexagons that tile the plane in a honeycomb lattice, then the graph formed by connecting those stations in adjacent hexagons is called $\Gamma_\Delta$. This configuration allows the use of few transmitters to cover a large amount of space, and one can see that the graph problem is in fact the same in this case. 

\begin{figure}[h!]
\centering
\begin{tikzpicture}[scale=2,auto=left,node distance = 2cm, every node/.style={circle,draw}]
 \tikzstyle{transmit} = [circle,fill = blue!20, draw];
  	\node[transmit] (n1) at (1, 20) {};
	\node[minimum size = 3cm] (n2) at (n1) {};
	\node[minimum size = 5cm, draw = red] (n2) at (n1) {};

	\node[transmit] (n2) [below = .7cm of n1] {};
	
	\node[transmit] (n3) [right = 1cm of n2] {};
	
	\node[transmit] (n4) [above right = .7cm of n3] {};


\draw	(n1) -- (n2);
\draw	(n2) -- (n3);
\draw	(n3) -- (n4);

 \end{tikzpicture}
\caption{Very close is depicted by a black circle or an edge. Close is the red circle.}\label{transmitter}

\end{figure}
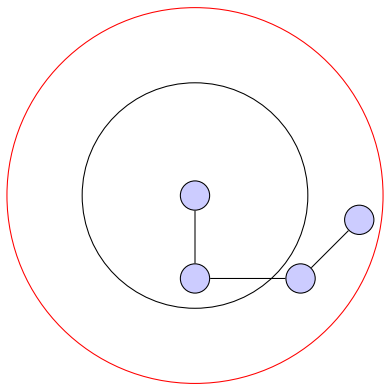

In fact, there are ways to get around the problem in Figure \ref{transmitter}. Instead of using the notion of graph distance, one might construct a graph where very close vertices are connected by red edges and close edges are connected by blue edges. The labeling must assign labels differing by $2$ to vertices connected by red edges and labels differing by $1$ to those connected by blue edges. This ensures that all frequency conditions are obeyed. More generally, one could assign weights $w(x,y)$ to all edges  $xy$ and attempt to label the vertices with a function $\phi: V(G) \rightarrow \{0, 1, ..., t\}$ such that $|\phi(x) - \phi(y)| \geq w(xy)$ for all vertices $x$ and $y$ of $G$. This problem is called the \emph{weighted graph model} of the channel assignment problem, and one can see that the $L(2,1)$-labeling problem is a special case of this problem. The weighted graph model is treated in the survey of McDiarmid \cite{mcdiarmid}, and a real number version is discussed by Griggs and Jin \cite{griggsjin} and Griggs and Kr\'{a}l \cite{griggskral}.

The $L(2,1)$-labeling problem has been studied with the central goal of finding bounds on the minimum span, or $\lambda_{2,1}(G)$. $\lambda_{2,1}(G)$ is the smallest number such that there exists an $L(2,1)$-labeling of $G$ with the difference $\lambda_{2,1}(G)$ between the highest and lowest label. $\lambda_{2,1}(G)$ is sometimes written $\lambda_{2,1}$ if there is no possibility for confusion. One parameter affecting the difficulty of labeling a graph $G$ is its maximum degree $\Delta(G)$. The maximum degree measures how many transmitters are close to one another, so large maximum degree should force labelings with large span. Much of the work on $L(2,1)$-labelings has consisted of attempts to bound $\lambda_{2,1}(G)$ in terms of the maximum degree. This work was galvanized by the early ``$\Delta^2$ Conjecture" of Griggs and Yeh \cite{griggs92}:

\begin{conj} If $\Delta(G) \geq2$, then $\lambda_{2,1} \leq \Delta^2.$\end{conj}

Several efforts have been made towards this bound \cite{griggs92}, \cite{chang96}, \cite{gon2005}, \cite{havet2012griggs}. This conjecture is the main focus of this thesis. 

\section{Terminology}\label{term}

\begin{defin}A \emph{graph} $G$ is a set of vertices $V(G)$ and a set of edges $E(G)$. Each edge is a set of two vertices. The edge consisting of the vertices $x$ and $y$ is denoted $xy$, and if $xy$ is in $E(G)$ then $x$ and $y$ are said to be \emph{adjacent}. 
\end{defin}

\begin{defin}A \emph{simple graph} is a graph with no edges from a vertex to itself, and at most one edge between any two vertices $x$ and $y$. \end{defin}

\noindent Throughout this paper the graphs we discuss are assumed to be simple and finite. 

\begin{defin}The \emph{degree} of  a vertex $v$ in $V(G)$ is the number of vertices adjacent to $v$. The degree of $v$ is denoted $d_{G}(v)$ or just $d(v)$ when the usage is clear from context. The set of vertices adjacent to $v$ is called the \emph{neighborhood} of $v$, denoted $N_G(v)$ or $N(v)$, and its members are called \emph{neighbors} of $v$.\end{defin}

\begin{defin} A \emph{proper coloring} of $G$ is a function $f$ from $V(G)$ to a set $S$ such that adjacent vertices map to different elements of $S$. 
\end{defin}

\begin{defin} A \emph{path} in a graph $G$ is an ordered list of distinct vertices $(v_1, ..., v_n)$ such that $v_i v_{i+1} \in E(G)$ for $1\leq i \leq n-1$. A \emph{cycle} is a path with $v_1 v_n \in E(G)$. \end{defin}

\begin{defin} A \emph{hamilton cycle} is a cycle using all the vertices of $G$.  A \emph{hamilton path} is a path using all the vertices of $G$. $G$ is said to be \emph{hamiltonian} if it contains a hamilton cycle. \end{defin}

\begin{defin} The \emph{distance} between two vertices $x$ and $y$ in $E(G)$, denoted $d_G(x,y)$ or $d(x,y)$, is the minimum length of any path from $x$ to $y$ in $G$. If there is no such path, we say $d(x,y) = \infty$. The \emph{diameter} of a graph is the maximum distance between any two vertices.\end{defin}

\begin{defin} The \emph{square} of a graph $G$, denoted $G^2$, is a graph with $V(G^2) = V(G)$ and $E(G^2) = \{xy\: |\: 0 <d(x,y) \leq 2\}.$ If a graph is diameter 2, then $G^2$ contains edges between all distinct vertices. \end{defin}

\section{Bounds and Results for the $L(2,1)$-labeling Problem}\label{bounds}

There have been many results on the $L(2,1)$-labeling problem, and this is by no means a comprehensive survey. For more information we direct the reader to \cite{yeh06}, \cite{cal06}, \cite{griggskral}, or \cite{griggs92}.

In their original 1992 paper, Griggs and Yeh bounded $\lambda_{2,1}$ for paths, cycles, trees. We will state the theorems here and include the occasional proof for the purpose of exposition. 

\begin{thm*}[Griggs, Yeh \cite{griggs92}]\label{paths} 
Let $P_n$ be a path on $n$ vertices. Then (i) $ \lambda_{2,1}(P_2) = 2$, (ii) $\lambda_{2,1}(P_3) = \lambda_{2,1}(P_4) = 3$, and (iii) $\lambda_{2,1}(P_n) = 4$ for $n \geq 4$.
\end{thm*}

\begin{figure}[h!]
\caption{$L(2,1)$-labelings of $P_2$, $P_3$, and $P_4$.}\label{pathfigure}
\centering
\vspace{1cm}
 \begin{tikzpicture}[scale=.8,auto=left,node distance = 2cm, every node/.style={circle,draw}]
  	\node (n1) at (1, 20) {0};
	\node (n2) [right of =n1] {2};
	\node (n3) [below of =n1] {1};
	\node (n4) [right of =n3] {3};
	\node (n5) [right of =n4] {0};
	\node (n6) [below of =n3] {1};
	\node (n7) [right of =n6] {3};
	\node (n8) [right of =n7] {0};
	\node (n9) [right of =n8] {2};

\draw	(n1) -- (n2);
\draw	(n3) -- (n4);
\draw	(n4) -- (n5);
\draw	(n6) -- (n7);
\draw	(n7) -- (n8);
\draw	(n8) -- (n9);

 \end{tikzpicture}
\end{figure}

\begin{proof}
For $n \leq 4$, appropriate labelings are shown in Figure ~\ref{pathfigure}. It is clear that 2 is the best possible span for $P_2$. Suppose $P_3$ were labelled with span 2; then there must be a vertex labeled 0 and one labeled 2. Then in any case, the remaining vertex cannot be labeled 1 due to the distance conditions. $P_4$ contains $P_3$, so it is also labeled with best possible span. Let $n \geq 5$ and index the vertices in the path, in order, by $i$ for $i \in \{1, ...., n\}$. One can see that the labeling 

\begin{displaymath}
   f(v_i) = \left\{
     \begin{array}{lr}
       0 & \textrm{ if } i \equiv 1 \pmod 5\\
       3 & \textrm{ if } i \equiv 2 \pmod 5\\
       1 & \textrm{ if } i \equiv 3 \pmod 5\\
       4 & \textrm{ if } i \equiv 4 \pmod 5\\
       2 & \textrm{ if } i \equiv 0 \pmod 5\\
       
     \end{array}
   \right.
\end{displaymath} 
suffices. We leave it to the reader to see that $P_5$ cannot be labeled with $\{0,1,2,3\}$.\end{proof}

\begin{figure}[b!]\label{pathfig}
\caption{$L(2,1)$-labelings of $P_2$, $P_3$, and $P_4$.}
\centering
\vspace{1cm}
 \begin{tikzpicture}[scale=.8,auto=left,node distance = 2cm, every node/.style={circle,draw}]
  	\node (n1) at (1, 20) {0};
	\node (n2) [right of =n1] {2};
	\node (n3) [below of =n1] {1};
	\node (n4) [right of =n3] {3};
	\node (n5) [right of =n4] {0};
	\node (n6) [below of =n3] {1};
	\node (n7) [right of =n6] {3};
	\node (n8) [right of =n7] {0};
	\node (n9) [right of =n8] {2};

\draw	(n1) -- (n2);
\draw	(n3) -- (n4);
\draw	(n4) -- (n5);
\draw	(n6) -- (n7);
\draw	(n7) -- (n8);
\draw	(n8) -- (n9);

 \end{tikzpicture}
\end{figure}

\begin{thm}[Griggs, Yeh \cite{griggs92}]\label{cycles} Let $C_n$ be a cycle on $n$ vertices. Then $\lambda_{2,1}(C_n) = 4$.
\end{thm}

Note that the above two theorems imply that the $\Delta^2$ Conjecture holds for $\Delta = 2$. The next theorem concerns special graphs called trees. A \emph{tree} is a graph with no cycles. Family trees and decision trees are two examples.

\begin{thm*}[Griggs, Yeh \cite{griggs92}]\label{trees} Let $T$ be a tree. Then $\lambda_{2,1}(T)$ is either $\Delta + 1$ or $\Delta + 2$. 
\end{thm*}

In fact, it is so difficult to tell whether $\lambda_{2,1}(T)$ is $\Delta + 1$ or $\Delta + 2$ that Griggs and Yeh conjectured that this decision problem is NP-Complete \cite{griggs92}. However, in 1995, Chang and Kuo showed that this not the case if $P\neq NP$.

\begin{thm*}[Chang, Kuo \cite{chang96}]
Let $T$ be a tree. There is a polynomial algorithm to decide if $\lambda_{2,1}(T) = \Delta +1$.

\end{thm*}
In addition, Griggs and Yeh proved that it is NP-complete to decide if a graph $G$ has $\lambda_{2,1}(G) \leq |V(G)|$ \cite{griggs92}. They conjectured that deciding if $\lambda_{2,1}(G) \leq k$ is $NP$-complete in general. Indeed, this was confirmed in 2001 with the following result.

\begin{thm*}[Fiala, Kloks, Kratchov\'{i}l \cite{fiala01}]
For each $\lambda\geq4$, it is NP-complete to decide if the input graph has $\lambda_{2,1}(G) \leq \lambda$.
\end{thm*}
For $\lambda \leq 3$, first check if $G$ contains a cycle. If it does, then $\lambda_{2,1}(G) \geq 4$ by Theorem \ref{cycles}. It is known that this check can be done in polynomial time \cite{tucker1984}. If $G$ does not contain a cycle, then $G$ is a tree, so $\lambda_{2,1}(G)$ can be determined in polynomial time by the theorem of Chang and Kuo. In any case, the decision problem is polynomial for $\lambda \leq 3$. 

Now we review some bounds on $\lambda_{2,1}$ for general graphs. The original bound in 1992 was obtained by first fit labeling techniques.
\begin{thm*}[Griggs, Yeh \cite{griggs92}]\label{firstbound}
Let $G$ be a graph. Then $\lambda_{2,1}(G) \leq \Delta^2+ 2\Delta$.
\end{thm*}
The next improvement on the general bound, proven in 1995, employed an algorithm which was used in the proof of the subsequent bound as well. 

\begin{thm*}[Chang, Kuo \cite{chang96}]\label{changbound}
Let $G$ be a graph. Then $\lambda_{2,1}(G) \leq \Delta^2+ \Delta$.
\end{thm*}

Before we begin the proof of the theorem, we need a definition and a fact. 
\begin{def}\label{indep}
A subset $S$ of a graph $G$ is \emph{independent} if no two vertices of $S$ are adjacent. 

\end{def}

\begin{fact}\label{degree} $\Delta(G^2) \leq \Delta(G)^2$. \end{fact}
\begin{proof}
Let $v \in G$. $v$ has at most $\Delta$ neighbors, and each of those have at most $\Delta - 1$ neighbors other than $v$. Hence, the total number of vertices at distance $2$ from $v$ is at most $\Delta + \Delta(\Delta - 1) = \Delta^2$.
\end{proof}

\begin{proof}[Proof of Chang and Kuo's bound]

The labeling scheme is iterative. Let $S_{-1} = \emptyset$. For $i \geq 0$, if a set $S_i$ is determined, label all its vertices with its index $i$. Let $F_i$ be the set of \emph{unlabeled} vertices at distance at least two from any vertices in $S_i$. Let $S_{i+1}$ be a maximal subset of $F_i$ that is independent in $G^2$, meaning one cannot add any vertices to $S_{i+1}$ and preserve independence in $G^2$. Then the vertices of $S_{i+1}$ will be labeled with $i+1$, and so on. Continue this process until all vertices are labeled.\\

This process must terminate, as if not all vertices are labeled and $S_i = \emptyset$, then $S_{i+1} \neq \emptyset$ because $F_i \neq \emptyset$. Now suppose that $v$ is a vertex with the maximum label $k$. One must wonder why $v$ was not put in the class $S_i$ for $0 \leq i < k$. The two possible reasons are that $v$ is adjacent to a vertex in $S_{i-1}\cup S_{i }$, or that $v$ is at distance $2$ from a vertex in $S_i$. The largest number of classes forbidden to $v$ due to these reasons is $$2d_G(v) + (\textrm{The number of vertices at distance two from } v).$$
From the proof of Fact \ref{degree}, this number is less than $2\Delta + \Delta(\Delta -1) = \Delta^2 + \Delta$. This means there are at most $\Delta^2 + \Delta + 1$ classes $S_0, ..., S_k$, so $k$ is at most $\Delta^2 + \Delta$. \end{proof}

The bound was further improved by Gon\c{c}alves in 2005, using a modified version of Chang and Kuo's algorithm. 

\begin{thm*}[Gon\c{c}alves \cite{gon2005}]
Let $G$ be a graph. Then $\lambda_{2,1}(G) \leq \Delta^2+ \Delta-2$.
\end{thm*}

For $\Delta = 3$, this gives $\lambda_{2,1}(G) \leq 10$, which is only one away from the bound of 9 required for the $\Delta^2$ Conjecture. The largest step towards the proof of the conjecture was made by Havet, Reed, and Sereni, who proved the conjecture for large $\Delta$. 

\begin{thm*}[Havet, Reed, and Sereni \cite{havet2012griggs}] $\lambda_{2,1}(G) \leq \Delta^2$ for all graphs with $\Delta$ larger than some $\Delta_0 \approx 10^{69}$. Consequently, $\lambda_{2,1}(G) \leq \Delta^2 + C$ for some absolute constant $C$. 
\end{thm*} 

One might wonder when $\Delta^2$ conjecture is tight. If the graph $G$ is diameter 2, i.e. all vertices are at distance less than two from one another, then $\lambda_{2,1}(G) \geq |V(G)|-1$ because all labels must be distinct. Griggs and Yeh showed that the $\Delta^2$ Conjecture holds for diameter two graphs. The maximum order of a diameter two graph is $\Delta^2 + 1$, so the conjecture must be tight for any diameter two graph of maximum size. There are either three or four such graphs: the 5-cycle, the Petersen Graph, the Hoffman-Singleton graph that has $\Delta = 7$, and possibly one more with $\Delta = 57$. These graphs are called the \emph{Moore Graphs} \cite{hoffman}. According to \cite{griggs92}, if $\Delta \geq 3$ and  $|V(G)| <\Delta^2 + 1$, then $\lambda_{2,1}(G) < \Delta^2$. Therefore, $\lambda_{2,1} \leq \Delta^2 - 1$ for diameter two graphs except for $C_3$, $C_4$ and the Moore Graphs. We have some comments about this bound in Section \ref{tightnesscomments}.

\section{Weighted Graph Model}\label{weighted}

Recall the definition from Section \ref{back} of the weighted graph model. We will use this model to define a slight generalization of the $L(2,1)$ labeling problem.

\begin{defin} 
Given a graph $G = (V,E)$ and a weight function $w: E \rightarrow \mathbb{Z}$, a labeling $\phi: V \rightarrow \{0, 1, ...., t\}$ is \emph{feasible} if and only if $$|\phi(u) - \phi(v)| \geq w(uv)$$ for all $uv \in E$.
\end{defin}
\begin{defin}
\noindent $Span(G, w)$ is the least integer $t$ for which a feasible labeling is possible. \end{defin}
Here we define a restriction of this problem, which we call the $(G, H)$-labeling problem. 

\begin{defin}Given a graph $G$ and a subgraph $H \subset G$, the \emph{$(G, H)$-labeling problem} is to find a feasible labeling of the pair $(G, w)$ with 

$$w(uv) = \left\{ 
\begin{array}{lr} 
2 \textrm{ if } uv \in E(H)\\
1 \textrm{ if } uv \in E(G)\setminus E(H) 
 \end{array}
 \right.
$$
$Span(G, H)$ is $Span(G,w)$ for $w$ as in the definition of the $(G, H)$-labeling problem.
\end{defin}
In other words, we only consider edge weights of 2 and 1. The edges of weight 2 form the subgraph $H$. Note that if $(G , H) = (F^2, F)$ for a graph $F$, then an $L(2,1)$-labeling of $F$ is equivalent to a $(F^2, F)$ labeling of $F^2$, so $\lambda_{2,1}(F) = Span(F^2, F)$. Also note that the set of pairs $(G, H)$ where $G$ and $H$ are finite simple graphs with $\Delta(G) \leq\Delta(H)^2$ and $|V(G)| = n$ contains the set of pairs $(F^2, F)$ where $F$ is a finite simple graph with $|V(F)| =n$.   Hence, we have the following simple result.

\begin{lem}\label{equiv} If $Span(G, H) \leq g(\Delta(H))$ for all pairs of graphs $(G, H)$ with $H \subset G$, $|V(G)| = n$ and $\Delta(G) \leq \Delta(H)^2$, then $\lambda_{2,1}(G) \leq g(\Delta(G))$ for all graphs $G$ of order $n$.
\end{lem}
This will allow us to consider only the more general $(G, H)$-labeling problem in searching for upper bounds. Often proofs of bounds, such as that of Chang and Kuo's bound, only rely on degree conditions. These are cases in which the technique could be extended to the $(G, H)$-labeling problem.\\

In addition, the $(G, H)$-labeling problem can model systems of transmitters in a more accurate way than the $L(2, 1)$-labeling problem. In the $L(2, 1)$-labeling problem, two transmitters can be close but not very close, yet they may still not be at distance two in $G$. For example, a system containing only two transmitters that are close but not very close would have no difference requirement for the frequency. However, in the $\Delta(G, H)$ problem these could be connected by an edge of $G$ at will.

\section{Hamiltonicity}\label{ham}
\begin{figure}[t!]

\begin{center}
 \begin{tikzpicture}[scale=2,auto=left,node distance = 2cm, thick, every node/.style= {circle,fill = blue!20, draw}]
\draw (18:2cm) -- (90:2cm) -- (162:2cm) -- (234:2cm) --
(306:2cm) -- cycle;
\draw (18:1cm) -- (162:1cm) -- (306:1cm) -- (90:1cm) --
(234:1cm) -- cycle;

\foreach \x in {18,90,162,234,306}{
\draw (\x:1cm) -- (\x:2cm);}
\node (1) at (18:2cm) {0};
\node (2) at (90:2cm) {3};
\node (3) at (162:2cm) {1};
\node (4) at (234:2cm) {4};
\node (5) at (306:2cm) {2};

\node (6) at (18:1cm) {5};
\node (7) at (90:1cm) {6};
\node (8) at (162:1cm) {7};
\node (9) at (234:1cm) {8};
\node (10) at (306:1cm) {9};

\draw [draw = red] (1) -- (3);
\tikz \fill[fill = red] (1);
\draw [draw = red] (3) -- (5);
\draw [draw = red] (5) -- (2);
\draw [draw = red](2) -- (4);
\draw [draw = red] (4)--(6);
\draw [draw = red] (6)--(7);
\draw [draw = red] (7)--(8);
\draw [draw = red] (8)--(9);
\draw [draw = red] (9)--(10);

\end{tikzpicture}  
\end{center}

\caption{An $L(2,1)$-labeling of the Petersen graph and a red hamilton path in the complement.}\label{peterham}

\end{figure}
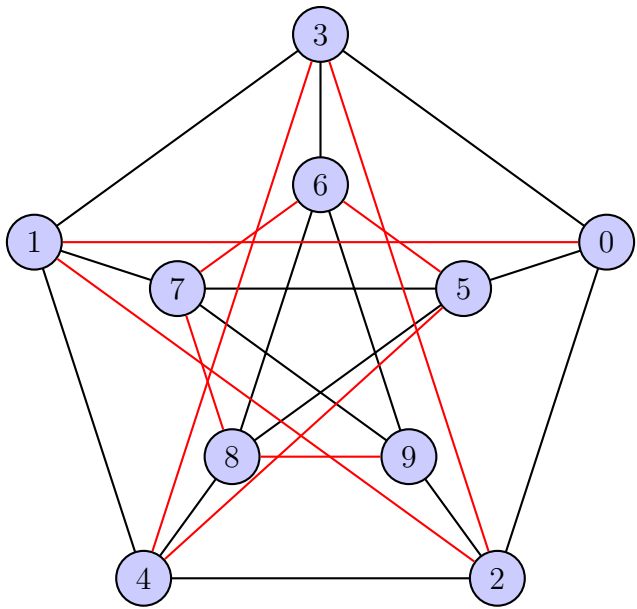

First we include a lemma of Griggs and Yeh to motivate the use of hamilton paths in $L(2,1)$-labelings. 

\begin{lem}[Griggs, Yeh \cite{griggs92}]\label{hamgrigg} There exists an injective $L(2,1)$-labeling of a graph $G$ with span $|V(G)|-1$ if and only if the complement of $G$ has a hamilton path.
	\end{lem}
	
\begin{proof}
$ $
\begin{itemize}
\item[($\Leftarrow$)] Suppose there exists a hamilton path $P$ in $G^c$. If the path's $i^{th}$ vertex is $p_i$, then form a labeling $f$ where $f(p_i) = i - 1$. Clearly the span of $f$ is $|V(G)|-1$. As the labeling is injective, no two vertices receive the same label, so the distance two condition is automatically met. If two vertices are adjacent in $G$, they are by definition \emph{not} adjacent in the path, so their labels differ by at least 2.
\item [($\Rightarrow$)] Suppose there is an injective $L(2,1)$-labeling of a graph $G$ with span $|V(G)|-1$. This is a numbering of the vertices with $\{0, 1, ..., |V(G)| - 1\}$. Form the path by the following rule: the $i^{th}$ vertex of the path $P$ is the vertex labeled $i-1$. Clearly $P$ has $|V(G)|$ vertices. In addition, for $0 \leq i <|V(G)$,  $p_i p_{i+1}$ cannot be in $E(G)$ by the definition of the labeling. Hence $p_i p_{i+1} \in E(G^c)$, so $P$ is a hamilton path.

\end{itemize}
An example is shown in Figure \ref{peterham}.
\end{proof}

The following generalization allows us to use the hamilton path technique to prove results about non-injective $(G,H)$-labelings. We will require a definition, of which an example is depicted in Figure \ref{adjgraph}.

\begin{defin}
Let $(G, H)$ be a pair of graphs with $H\subset G$, and let $C = \{C_0, C_1, ...., C_n\}$ be a proper coloring of $G$. 
	The \emph{color adjacency graph} of the pair $(C,G,H)$, denoted \emph{$\mathcal{CGH}$}, is a new $n$-vertex graph with
	
	$$ V(\mathcal{CGH}) = C$$
	 and
	$$E(\mathcal{CGH}) = \{C_i C_j \textrm{ such that there is an edge of } H \textrm{ between } C_i \textrm{ and } C_j\}.$$

\end{defin}

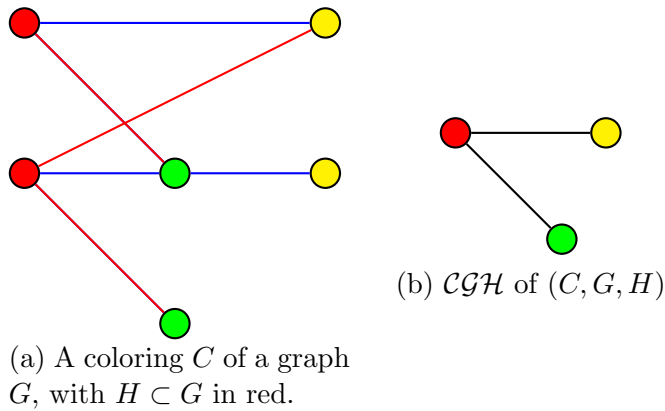
\begin{figure}[h!]
\caption{Red edges are weighted 2 and blue ones are weighted 1.}\label{adjgraph}

\centering
\begin{subfigure}[b!]{0.3\textwidth}
 \begin{tikzpicture}[scale=2,auto=left,node distance = 2cm, thick]
 \tikzstyle{node} = [circle,fill = blue!20,draw];
 \tikzstyle{textnode}=[];
  \node[node][fill = red] (n1) at (2,20) {};
 \node[node][fill = red] (n2) [below of=n1] {};
 
\node[node][fill = green] (n3) [right of = n2] {};
\node[node][fill = green] (n4) [below of = n3] {};
 
\node[node][fill = yellow] (n5) [right of  = n3] {};
 \node[node][fill = yellow] (n6) [above of = n5] {};
 

\foreach \from/\to in {n1/n3, n2/n4, n2/n3, n3/n5, n1/n6}{
 	\draw[draw = blue] (\from) -- (\to);}
	
\foreach \from/\to in {n1/n3, n2/n4, n2/n6}{
	\draw[draw = red] (\from) -- (\to);}
\end{tikzpicture}
\subcaption{A coloring $C$ of a graph $G$, with $H\subset G$ in red.}
\end{subfigure}
\begin{subfigure}[b!]{0.3\textwidth}
\centering
 \begin{tikzpicture}[scale=2,auto=left,node distance = 2cm, thick] 
 \tikzstyle{node} = [circle,fill = blue!20,draw];
 
 \node[node][fill = red] (n7) {};
 
\node[node][fill = green] (n8) [below right of = n7] {};
 
 \node[node][fill = yellow] (n9) [right of = n7] {};
 

 \foreach \from/\to in {n7/n8, n7/n9}{
 	\draw (\from) -- (\to);}

\end{tikzpicture}  

\subcaption{$\mathcal{CGH}$ of $(C,G,H)$}
\end{subfigure}

\end{figure}

The next lemma is an equivalent formulation of the $(G, H)$ labeling problem based on the color adjacency graph. 

\begin{lem}
$Span(G,H) \leq L - 1$ if and only if there exists a coloring $C$ of $G$ with $L$ classes, which can be empty, such that complement $\mathcal{CGH}^c$ of the color adjacency graph of $(C,G,H)$ has a hamilton path.

\end{lem}

\begin{proof}We begin with the backward direction. By assumption, $\mathcal{CGH}^c$ has a hamilton path $P = \{p_0, p_1, ...., p_{l-1}\}$. Recall that the vertices of $P$ are color classes partitioning $G$. Let $f:V(G)\rightarrow \mathbb{Z}$ be defined as $f:v \mapsto i$ where $i$ is the unique index such that $v \in p_i $. We now check that $f$ is a $(G,H)$-labeling. If $xy \in G$, then $x$ and $y$ are given two different labels because $C$ is a coloring of $G$. If $xy \in H$, then $x$ and $y$ are in two distinct color classes $p_i$ and $p_j$ such that $p_ip_j \in E(\mathcal{CGH})$. Then $p_ip_j \notin E(\mathcal{CGH}^c)$, so $i \neq j \pm 1$ because otherwise $p_ip_j \in E(P)$. Therefore $|f(x) - f(y)| \geq 2$, and $f$ is an $(G,H)$-labeling for $G$. \\

\noindent For the forward direction, we assume that there is a $(G,H)$-labeling $f$ of $G$ with span at most than $L - 1$. For $0 \leq 0 <L$, let $P_j$ be the set of vertices labeled $j$. If $j$ not used as a label, add an empty class $P_j$. Observe that $P = \{P_0, ...., P_{L-1}\}$ is a coloring of $G$ with $L$ classes, and is in fact a hamilton path in $\mathcal{PGH}^c$. This is because for $0 \leq i \leq L-2$, $P_iP_{i+1}\neq \mathcal{PGH}^c$ means there is an edge of $H$ between $P_i$ and $P_{i+1}$. This implies that two vertices labeled $i$ and $i+1$ are adjacent in $H$, a contradiction to the fact that $f$ is a $(G,H)$-labeling.\end{proof} 

In order to use the above formulation of the problem to find labelings, it is necessary to prove the existence of hamilton paths. For this we will present a few theorems relating the degrees of the vertices in a graph to its hamiltonicity. We include the proofs for completeness.

\begin{thm}[Bondy and Chv\'{a}tal \cite{chvatal1972}]\label{bondy} Let $G$ be a graph with $n \geq 3$ vertices. If $uv \notin E(G)$, $u \neq v$, and $d(u) + d(v) \geq n$, then $G$ has a hamilton cycle if and only if $G + uv$ does. 
\end{thm}

\begin{proof} If $n \leq 2$, then  $uv \notin E(G)$, $u \neq v$, and $d(u) + d(v) \geq n$ is impossible. If $G$ has a hamilton cycle, then $G+uv$ clearly does because adding an edge will take away no edge of the cycle. If $G + uv$ has a hamilton cycle, $G$ has a hamilton path $u p_1p_2 ... p_{n-2} v$ that starts at $u$ and begins at $v$. Consider the neighborhood of $u$, $N(u) \subset \{p_1, ...., p_{n-2}\}$ .  Define $P(u)$ to be the set $\{p_i \textrm{ such that } p_{i+1} \in N(u)\}$. One can see that $P(u)-u \subset  \{p_1, ...., p_{n-3}\}$, and that $|P(u) - u| = |N(u)|-1$.  Note that $N(v), P(u)-u \subset G - v - u$, which has cardinality $n-2$. Now 
\begin{eqnarray*}|N(v)\cap (P(u) - u)|  = |N(v)| + |P(u)-u|- |N(v) \cup (P(u)-u)| \\
\geq |N(v)| + |P(u)-u|- (n-2) =  |N(v)| + |N(u)| -1- (n-2) \geq 1.
\end{eqnarray*}
This means that there is a vertex $p_i$ in $P(u) \cap N(v)$. Then $u p_{i+1} p_{i+2}....v p_{i} p_{i-1} ....p_1 u $ is a hamilton cycle in $G$. 
\end{proof}

\begin{thm}[P\'{o}sa, \cite{kronk}]\label{posathm} Let $G$ have $n \geq 3$ vertices. If for every $k$, $1\leq k \leq (n-1)/ 2 $, \\
$|\{v : d(v) \leq k \}| < k$, then $G$ contains a hamilton cycle. 
\end{thm}

\begin{proof}
Suppose $G$ does not have a hamilton cycle. Connect nonadjacent vertices one at a time until no edge can be added without creating a hamilton cycle. Call this new graph $\tilde{G}$. The hypotheses of Theorem \ref{posathm} are still satisfied by $\tilde{G}$ , as we only increased degrees. By Theorem \ref{bondy}, now all all nonadjacent distinct vertices $u$ and $v$ have $d(u) + d(v) < n$. Choose a pair of distinct nonadjacent vertices $u,v$ with $d(u) + d(v)$ largest possible. This can be done because $\tilde{G}$ is not a complete graph. \\
Like in the proof of Theorem \ref{bondy}, there is a hamilton path $up_1....p_{n-2}v$. One of $u$ and $v$ must have degree $k$ less than or equal to $n-1/2$, so suppose without loss of generality that this vertex is $u$. $k$ is not zero by the hypothesis. Allow $P(u)$ to be the same as in the proof of Theorem \ref{bondy}. No vertex in $P(u)$ is adjacent to $v$, or there would be a hamilton cycle. Also, if $x \in P(u)$ then $d(x) \leq k$ by maximality of $d(u) + d(v)$. But $|P(u)| = |N(u)| = d(u) = k$; now there are $k$ vertices of degree at most $k$, a contradiction.  \end{proof}

The following corollary is used to translate these theorems into facts about hamilton paths rather than cycles. The proof follows a classic trick. 

\begin{cor}\label{posacor}
Let $G$ have $n$ vertices. If for every $k$, $0\leq k \leq (n-2)/ 2 $, \\
$|\{v : d(v) \leq k \} | \leq k$, then $G$ has a hamilton path. 
\end{cor}

\begin{proof}The corollary follows easily by adding a dominating vertex to $G$ and observing that by Posa's Theorem the new graph is hamiltonian. Removing the dominating vertex leaves a hamilton path in the original graph.\end{proof}

\section{Main Results}\label{main}

In order to use theorems about hamilton paths on the color adjacency graph, we will need to control its maximum degree. We can do this by controlling the size of the color classes, which we can accomplish by the use of a powerful result of Szemer\'{e}di and Hajnal on equitable colorings \cite{hajnal}.

\begin{thm}[Hajnal, Szemer\'{e}di, \cite{hajnal}, \cite{kierstead}, \cite{shortpf}]\label{hajnal} If $\Delta(G) \leq r$, then $G$ can be equitably colored with $r +1$ colors; that is, the sizes of the color classes differ by at most one. \end{thm}

We now present our main result. 

\begin{thm}\label{main}Let $(G, H)$ be a pair of graphs with $H \subset G$, and let $L$ be an integer. Let $\Delta = \Delta(H) \geq 1$.
Suppose the following hold:
\begin{itemize}
\item[(i)] $\Delta(G) \leq \Delta^2,$ and
\item[(ii)] $L\geq \Delta^2 + 1.$

\end{itemize}
Then $Span(G,H) \leq L - 1$ if 

$$|V(G)| \leq (L - \Delta)\left(\left\lfloor\frac{L-1}{2\Delta} \right\rfloor + 1\right) - 1.$$
\end{thm}

Before the proof of Theorem \ref{main}, we will discuss two corollaries that have implications for the $\Delta^2$ Conjecture.

\begin{cor}\label{lambdacor} Let $G$ be a graph with $\Delta = \Delta(G) \geq 1$, and let $L$ be an integer with $L\geq \Delta^2 + 1$. Then $\lambda_{2,1}(G) \leq L - 1$ if 

$$|V(G)| \leq (L - \Delta)\left(\left\lfloor\frac{L-1}{2\Delta} \right\rfloor + 1\right) - 1.$$
\end{cor}
\begin{proof} This follows from Lemma \ref{equiv}.\end{proof}

\begin{cor}\label{maincor}
Let $G$ be a graph of with $\Delta = \Delta(G) \geq 1$. Then $\lambda_{2,1}(G) \leq \Delta^2$ if 

$$|V(G)| \leq \left(\left\lfloor \frac{\Delta}{2} \right\rfloor + 1 \right)\left(\Delta^2 - \Delta+ 1\right) - 1.$$

\end{cor}
\begin{proof}
Using Corollary \ref{lambdacor} with $L = \Delta^2 + 1$ gives the desired result. \end{proof}

Corollary \ref{maincor} significantly expands the known orders of graphs that satisfy the $\Delta^2$ Conjecture; it does so more dramatically as $\Delta(G)$ increases. For $\Delta(G) = 3$, $|V(G)| \leq 13$ suffices as opposed to the previously known $|V(G)|\leq10$ \cite{griggs92}. For $\Delta(G) = 4$, we have $|V(G)| \leq 38$ as opposed to $|V(G)| \leq 17$ \cite{griggs92}. If $G$ is the Hoffman-Singleton graph, then $\Delta(G) = 7$, $|V(G)| = 50 = \Delta^2 + 1$, and in fact $\lambda_{2,1}(G) = 49 = \Delta^2$ \cite{griggs92}. It might seem productive to look among minor variations of the Hoffman-Singleton graph for counterexamples to the $\Delta^2$ Conjecture, but Corollary \ref{maincor} suggests otherwise - the conjecture holds if $\Delta(G) = 7$ and $|V(G)| \leq 169$.  The bounds on $|V(G)|$ established in Corollary \ref{maincor} grow quickly with $\Delta$, as they are cubic in $\Delta$ rather than quadratic as in \cite{griggs92}. \\

For some $|V(G)|$, we can also use Theorem \ref{main} to find stronger upper bounds on $\lambda_{2,1}(G)$ than the best known bound of Gon\c{c}alves  \cite{gon2005}. The bound on $|V(G)|$ in the following corollary is larger than the bound in Corollary \ref{lambdacor}.

\begin{cor} Let $G$ be a graph with $\Delta = \Delta(G) \geq 3$. Then $\lambda_{2,1}(G) < \Delta^2 + \Delta - 2$ if

$$|V(G)| \leq \left(\left\lfloor \frac{\Delta}{2} \right\rfloor + 1 \right)\left(\Delta^2 - 2\right) - 1.$$

\end{cor}

\begin{proof}
Apply Corollary \ref{lambdacor} with $L = \Delta^2 + \Delta - 2$. This gives 
$$|V(G)| \leq \left(\left\lfloor \frac{\Delta}{2} + \frac{1}{2} -\frac{ 3}{2\Delta} \right\rfloor + 1 \right)\left(\Delta^2 - 2\right) - 1.$$

\noindent Since we have assumed $\Delta \geq 3,$ we have $0\leq  {1}/{2} -{ 3}/{(2\Delta)} < 1/2$, so 

$$\left\lfloor \frac{\Delta}{2} + \frac{1}{2} -\frac{ 3}{2\Delta} \right\rfloor = \left\lfloor \frac{\Delta}{2} \right\rfloor.$$ \end{proof}

For comparison, this bound is the same as Corollary \ref{maincor} for $\Delta = 3$. For $\Delta = 4$, we get $|V(G)| \leq 41$ as opposed to 38, and for $\Delta = 7$ we get $|V(G)| \leq 187$ as opposed to 169. \\

We now proceed to the proof of Theorem \ref{main}.

\begin{proof} 

Let $L$ be as in Theorem \ref{main}. We will show that for any integers $q \geq 0$, $0\leq r\leq L-1$ with

$$Lq + r \leq M =  (L - \Delta)\left(\left\lfloor\frac{L-1}{2\Delta} \right\rfloor + 1\right) - 1,$$

if $|V(G)| = Lq + r$, $\Delta(H) = \Delta$, and $\Delta(G) \leq \Delta^2$ then $(G, H)$ has a $(G,H)$-labeling with span at most $L-1$. This is sufficient to prove Theorem \ref{main}, as for any integer $n$ there exist unique integers $q\geq 0$ and $r \in \{0, ..., L-1\}$ with $Lq + r = n$. Suppose $|V(G)| = Lq+r$.  Recall that $L \geq \Delta^2 + 1 \geq \Delta(G) + 1$. By the Szemer\'edi-Hajnal theorem, $G$ has an equitable coloring $C$ with $L$ color classes. For convenience we will use all $L$ color classes even if several are empty.  This means $L - r$ classes have $q$ vertices and $r$ classes have $q + 1$ vertices. Our goal is to prove that the complement of the color adjacency graph of $(C,G,H)$, or $\mathcal{CGH}^c$, has a hamilton path. Note that $d_{\mathcal{CGH}}(V) \leq \Delta |V|$ for all $V \in V(\mathcal{CGH})$. Write the degree of $V$ in $\mathcal{CGH}^c$ as $d_c(V)$.\\
\begin{itemize}
\item[Case 1:] $q \leq \lfloor(L-1)/2\Delta \rfloor - 1$. \\
Then 
$$\Delta(q +1) \leq  \Delta \left\lfloor\frac{L-1}{2\Delta} \right\rfloor \leq  \left\lfloor\frac{L-1}{2}\right\rfloor$$

so that $\delta(\mathcal{CGH}^c) \geq L-1 - \lfloor(L-1)/2\rfloor \geq (L-1)/2$, and the conditions of Corollary \ref{posacor} are satisfied. Therefore $\mathcal{CGH}^c$ has a hamilton path.\\

\item [Case 2:] $q = \lfloor(L-1)/2\Delta \rfloor$.\\
Now $M = (L - \Delta)(q + 1) - 1$.
$r$ can only be as big as $M - Lq$, which is 
\begin{eqnarray*} Lq + L  - \Delta (q + 1) - 1 - Lq \\
  = L - 1 - \Delta(q + 1) \\
  = L - 1 - \Delta\left(\left\lfloor\frac{L-1}{2\Delta} \right\rfloor + 1\right) \geq 0.
 \end{eqnarray*}


Now suppose $k$ is an integer with $0 \leq k \leq (L - 2)/2$ as in Corollary \ref{posacor}. If $d_c(V) \leq k$, then 

$$\frac{L - 2}{2} \geq L- 1 - d_{\mathcal{CGH}}(V) \geq L- 1 - \Delta |V|, $$\\
\noindent so that $|V| \geq (1/\Delta)(L - 1 - (L - 2)/2) = (L - 1)/2\Delta + 1/2\Delta>q$. Therefore $|V| = q + 1$, so we know there are at most $r$ vertices with $d_c(V) \leq k$. For any such vertex $V$, 

$$d_c(V) \geq  L- 1 - (q + 1) \Delta = L - 1 - \Delta\left(\left\lfloor\frac{L-1}{2\Delta} \right\rfloor + 1\right) \geq r \geq 0.$$

\noindent Now the conditions of Corollary \ref{posacor} are satisfied, so $\mathcal{CGH}^c$ has a hamilton path.
 As

$$Lq + L - 1 - \Delta\left(\left\lfloor\frac{L-1}{2\Delta} \right\rfloor + 1\right)  = (L - \Delta)\left(\left\lfloor\frac{L-1}{2\Delta} \right\rfloor + 1\right) - 1 = M,$$

\noindent this argument works for any $|V(G)| \leq M$. \end{itemize}
From Lemma \ref{ham}, $\mathcal{CGH}^c$ having a hamilton path implies that $(G, H)$ has an $(G,H)$-labeling with $Span(G, H) = L-1$. \end{proof}

Some authors have been concerned with finding labelings that are surjective on a set of integers $\{0, ..., t\}$ \cite{georges2005structure}. Such labelings are said to be no-hole.

\begin{cor}
Let $G$ be a graph of order $n$ with $\Delta = \Delta(G) \geq 1$, and let $L$ be an integer with $L\geq \Delta^2 + 1$. If 

$$n \leq (L - \Delta)\left(\left\lfloor\frac{L-1}{2\Delta} \right\rfloor + 1\right) - 1,$$
then there is an $L(2,1)$-labeling of $G$ with a span at most $L-1$ that is equitable. If $n\geq L$, the labeling is no-hole.

\end{cor}
\begin{proof}
The proof follows immediately from the proof of Theorem \ref{main}.
\end{proof}

The next corollary concerns algorithms involved in finding these labelings. 
\begin{cor}

Let $G$ be a graph of order $n$ with $\Delta = \Delta(H) \geq 1$, $\Delta(G) \leq \Delta(H)^2$, and $L\geq \Delta^2 + 1$. There is an algorithm with polynomial running time in $n$ to compute an $(G ,H)$-labeling of with span at most $L-1$ for all $n$ and $L$ such that 
$$n \leq (L - \Delta)\left(\left\lfloor\frac{L-1}{2\Delta} \right\rfloor + 1\right) - 1.$$
\end{cor}

\begin{proof}
If $L \geq 2n+1$, the appropriate labeling can be obtained by labeling the vertices $0, 2, ..., 2n$ in any order \cite{griggs92}. This can clearly be done in polynomial time. Otherwise, in \cite{kierstead} there is shown to be an algorithm polynomial in $n$ to equitably color $G$ with $L$ colors. Degree sequences satisfying the conditions of P\'{o}sa's Theorem also satisfy those of Chv\'{a}tal's Theorem \cite{bondy76}, and the paper's authors exhibit an algorithm polynomial in $p$ to find hamilton cycles in graphs of order $p$ which satisfy the conditions of Chv\'{a}tal's Theorem. From the proof of Lemma \ref{ham} and of Corollary \ref{posacor}, we see that to find the labeling it is enough find a hamilton cycle in a certain graph, namely $\mathcal{CGH}^c$ with a dominating vertex added, of order $L + 1 \leq 2n + 2$ that satisfies the conditions of P\'{o}sa's Theorem. From \cite{bondy76}, we can do this with an algorithm polynomial in $L + 1$, which must also be polynomial in $n$. These two algorithms in succession yield the desired algorithm. \end{proof}

\section{Comments on Diameter Two Graphs}\label{tightnesscomments}

As mentioned in Section \ref{back}, we have the following result:

\begin{thm}[Griggs, Yeh \cite{griggs92}]\label{diam} The $\Delta^2$ Conjecture holds for diameter two graphs.
In addition, $\lambda_{2,1} \leq \Delta^2 - 1$ for diameter two graphs with $\Delta \geq 2$ except for $C_3$, $C_4$ and the Moore Graphs. For these exceptional graphs, $\lambda_{2,1} = \Delta^2$. \end{thm}

The proof of these facts rely on Lemma \ref{ham}, Brooks' Theorem, and another result from Griggs and Yeh. 

\begin{thm}[Brooks \cite{lovasz1975three}]$\chi(G) \leq \Delta + 1$, and $\chi(G) \leq \Delta$ unless $G$ is an odd cycle or a complete graph.
\end{thm}

\begin{lem}[Griggs, Yeh \cite{griggs92}]\label{chrom}
$\lambda_{2,1}(G) \leq |V(G)| + \chi(G) - 2.$

\end{lem}

\begin{proof}[Proof of Theorem \ref{diam}]
If $\Delta = 2$, one can verify the theorem readily. 
If $\Delta \geq 3$, we can split into cases: suppose $\Delta \geq (|V(G)|)/2$. Lemma \ref{chrom} implies that $\lambda_{2,1}(G) \leq 2 \Delta + \chi(G) -2$. If $G$ is a complete graph, then clearly $\lambda_{2,1}(G) = 2\Delta(G)$. $G$ is not an odd cycle, as $\Delta \geq 3$. Otherwise, $2 \Delta + \chi(G) -2 \leq 3\Delta - 2$ by Brooks' Theorem. Note that in both cases, $\Delta(G) \geq 3$ implies that $\lambda_{2,1}(G) \leq \Delta^2 - 2$. \\

On the other hand, suppose $\Delta \leq (|V(G)| - 1)/2$. Then $\delta(G^c) \geq (|V(G)| - 1)/2$. Also, $G$ has diameter 2, so $|V(G)| \geq 3$. By Corollary \ref{posacor}, $G^c$ has a hamilton path. By Lemma \ref{hamgrigg}, there is an $L(2,1)$-labeling of $G$ with span $|V(G)|-1$. As the Moore graphs are the only diameter two graphs with $|V(G)| = \Delta^2 + 1$, Theorem \ref{diam} holds.\end{proof}

In fact, we can do better by the following result:

\begin{thm} [Erd\H{o}s, Fajtlowicz, Hoffman \cite{erdosdiam}]
Except $C_4$, there is no diameter two graph of order $\Delta^2$.
\end{thm}
This and the proof of Theorem \ref{diam} imply the following theorem.
\begin{thm}With the exception of $C_3$, $C_4$, $C_5$ and the Moore Graphs, any diameter two graph with $\Delta(G) \geq 2$ has $\lambda_{2,1}(G) \leq \Delta^2 - 2$.
 \end{thm}
 
 We also have some comments on a special family of diameter two graphs that have $\lambda_{2,1}$ high. In order to do this, we must define the points of the \emph{Galois Plane}, denoted $PG_2(n)$. Let $F$ be a finite field of order $n$. Let $P = F^3\setminus \{(0,0,0)\}$. Define an equivalence relation $\equiv$ on $P$ by $(x_1, x_2, x_3) \equiv (y_1, y_2, y_3) \iff (x_1, x_2, x_3) = (cy_1, cy_2, cy_3) $ for some $c \in F$. The \emph{points} of $PG_2(n)$ are the equivalence classes. 
 
\begin{defin} The \emph{polarity graph of} $PG_2(n)$, denoted $H$, is the graph with the points of $PG_2(n)$ as vertices and with two vertices $(x_1, x_2, x_3)$ and $(y_1, y_2, y_3)$ adjacent if and only if $y_1x_1 + y_2x_2 + x_3 x_3 = 0$. \end{defin} 

By the properties of $PG_2(n)$, we know that the diameter of $H$ is two, $\Delta(H) = n+1$ and its order is $n^2 + n + 1 = \Delta^2 - \Delta + 1$ \cite{karteszi1976introduction}. This implies that $\lambda_{2,1}(H) \geq \Delta^2 - \Delta$. In fact, Yeh showed that $\lambda_{2,1}(H) = \Delta^2 - \Delta$ \cite{griggs92}. This is an infinite family of graphs, as finite fields exist for $n = p^k$ with $p$ prime.

However, we can improve this by one. This construction follows Erd\H{o}s, Fajtlowicz and Hoffman \cite{erdosdiam}. A vertex $(x,y,z)$ in $H$ has degree $n$ if and only if the norm $x^2 + y^2 + z^2 = 0$. Suppose $F$ has characteristic $2$ and the order of $F$ is $n$. If $(a,b,c)$ is in $H$ then it is adjacent to the point $(b+c, a + c, a + b)$ which has norm 0 and is also in $H$. In other words, every vertex in $H$ is adjacent to a vertex of degree $n$. We proceed to find the number of points of degree $n$ in $H$. Since the characteristic of $F = 2$, $f(x) = x^2$ is injective and hence surjective on $F$. This means we can choose $x^2$ and $y^2$ freely as as long as one of them is nonzero, and then $z^2$ is determined. We must also eliminate proportional pairs, so in total this leaves $(n^2 - 1)/(n-1) = n+1$ vertices of degree $n$. \\

Now we can make an $n+1$ regular, diameter two graph $\tilde{H}(n)$ by adding a vertex that is adjacent to all vertices of degree $n$. This graph is of order $n^2 + n + 2 = \Delta^2 -\Delta + 2$.

\begin{thm}
The graph $\tilde{H}(n)$ has $\lambda_{2,1}(\tilde{H}) = \Delta^2 - \Delta + 1$.
\end{thm}
\begin{proof} Because $\tilde{H}$ is diameter two, $\lambda_{2,1}(\tilde{H}) \geq \Delta^2 - \Delta + 1$. As $\Delta \geq 3$, \\
$\Delta \leq (\Delta^2 - \Delta + 1)/2 = (|V(H)| - 1)/2$. By the proof of Theorem \ref{diam}, $\lambda_{2,1}(\tilde{H}) \leq |V(G)| - 1 = \Delta^2 - \Delta + 1$. \end{proof}

$\tilde{H}(n)$ exists for all $n = 2^k$, so this is an infinite family of graphs.




\section{Future Work}\label{future}
Based on the work in this paper and the similarities between the $(G, H)$-labeling problem and the $L(2,1)$-labeling problem, it seems reasonable to make the following conjecture:

\begin{con}
Let $(G, H)$ be a pair of graphs with $H \subset G$. Let $\Delta = \Delta(H) \geq 2$, and suppose $\Delta(G) \leq \Delta^2.$ Then $$Span(G,H) \leq \Delta^2.$$
\end{con}

There may also be a way to raise the bound on $|V(G)|$ in Theorem \ref{main} slightly. In \cite{chvatal1972}, Chv\'{a}tal provides the best possible characterization of degree sequences forcing hamilton cycles. We believe that there may be some strengthening of the results in this paper based on Chv\'{a}tal's work.\\

A \emph{circular distance two labeling with span} $\mu$ is an $L(2,1)$-labeling of $G$ with $\{0, .... \mu\}$ with the added condition that if $xy$ in $G$ then $|f(x) - f(y)| \neq \mu$ \cite{liu2005circular}. This problem can be translated into the $(G, H)$-labeling setting by letting a $\mu(G, H)$-labeling be a $(G, H)$-labeling with $\{0, .... \mu\}$ with the added condition that if $xy$ in $H$ then $|f(x) - f(y)| \neq \mu$. It seems reasonable to assume that finding these labelings translates to finding a coloring $C$ of $G$ such that $\mathcal{CGH}^c$ contains a hamilton cycle. To this end, we make the following stab at an extension of this paper's results for $\mu(G, H)$-labelings. 

\begin{con}
Let $(G, H)$ be a pair of graphs with $H \subset G$, and let $\mu$ be an integer. Let $\Delta = \Delta(H) \geq 1$.
Suppose the following hold:
\begin{itemize}
\item[(i)] $\Delta(G) \leq \Delta^2,$ and
\item[(ii)] $\mu\geq \Delta^2 + 1.$
\end{itemize}
Then there exists a $\mu(G,H)$-labeling if 

$$|V(G)| \leq (\mu - \Delta)\left(\left\lfloor\frac{\mu-2}{2\Delta} \right\rfloor + 1\right) - 2.$$
\end{con}
The expression is motivated by the slight difference between P\'{o}sa's theorem for hamilton cycles and the modification for hamilton paths. The degree of the color adjacency graph has to be a half unit greater for P\'{o}sa's theorem about cycles to work, and there must be one fewer low degree vertex of each degree. After some algebra, this results in the expression in the conjecture.

\section{Acknowledgements}\label{ack}

I would like to thank my thesis advisor, Jerrold Griggs, who was a great help throughout the process. I also thank the other members of the committee, Linyuan Lu and L\'{a}zl\'{o} Sz\'{e}kely, and my academic advisor, Dr. \'{E}va Czabarka. Finally, I would like to thank my family and friends for their support. We all know it would be impossible without them.

\bibliographystyle{plain}
\bibliography{cheez}

\end{document}